\documentclass{amsart}

\usepackage{amssymb}
\usepackage{amsmath}
\usepackage{amsthm}
\usepackage{amsbsy}
\usepackage{bm}
\usepackage[english]{babel}
\usepackage{graphicx}
\usepackage{color}
\usepackage{graphicx}
\usepackage{color}
\usepackage{hyperref}

\theoremstyle{plain}
\newtheorem{theorem}{Theorem}[section]
\newtheorem*{theorem*}{Theorem}
\newtheorem{lemma}[theorem]{Lemma}
\newtheorem{corollary}[theorem]{Corollary}

\theoremstyle{definition}
\newtheorem{definition}[theorem]{Definition}

\newcommand{\Z}{\mathbb{Z}}
\newcommand{\R}{\mathbb{R}}
\newcommand{\g}{\gamma}
\renewcommand{\H}{{\mathbb{H}}}
\renewcommand{\t}{\tau}
\renewcommand{\a}{\alpha}
\renewcommand{\b}{\beta}
\newcommand{\first}{\g(\a^m,\b)}
\newcommand{\second}{\g(\a_1^m,\b_1)}

\newcommand{\la}{\langle}
\newcommand{\ra}{\rangle}

\theoremstyle{plain} 
\newcommand{\thistheoremname}{}
\newtheorem*{genericthm*}{\thistheoremname}
\newenvironment{style}[1]
  {\renewcommand{\thistheoremname}{#1}%
   \begin{genericthm*}}
  {\end{genericthm*}}

\begin{document}
\title[Center of Goldman Algebra]{Center of the Goldman Algebra}

\author{Arpan Kabiraj}

\address{	Department of Mathematics\\ 
		Indian Institute of Science\\
		Bangalore 560012, India}

\email{arpan.into@math.iisc.ernet.in}

\thanks{The author was supported by NET-CSIR (India) Senior Research Fellowship.}

\begin{abstract}
We show that the center of the Goldman algebra associated to a closed oriented hyperbolic surface is trivial. For a hyperbolic surface of finite type with nonempty boundary, the center consists of closed curves which are homotopic to boundary components or punctures. 
\end{abstract}
\maketitle
\section{Introduction}
Given an oriented surface $F$ and two oriented closed curves $\a$ and $\b$ intersecting transversely in double points, Goldman \cite{Gol} defined a bracket operation between $\a$ and $\b$ by $$[\a,\b]=\sum_{p\in\a\cap\b}\epsilon(p)(\a*_p\b)$$ where $\epsilon(p)$ denotes the sign of the intersection between $\a$ and $\b$ at $p$ and $(\a*_p\b)$ denotes the loop product of $\a$ and $\b$ at $p$.

In the paper \cite{Gol}, Goldman showed that the bracket operation is well defined on $\mathcal{C}$, the set of all free homotopy classes of oriented closed curves in $F$. Hence it can be extended linearly to $\Z(\mathcal{C})$, the free module generated by $\mathcal{C}$. He also proved that the bracket is skew-symmetric and satisfies Jacobi identity. Therefore the bracket operation is a Lie bracket and it gives a Lie algebra structure on $\Z(\mathcal{C})$ which we denote by $\mathcal{L}(F)$. The main object of this paper is to study the center of $\mathcal{L}(F)$.

Using algebraic tools Etingof \cite[Theorem 1.2]{E} proved the following conjecture of M. Chas and D. Sullivan. We give an alternative proof of this result using hyperbolic geometry.
\begin{style}{Theorem 1}
If $F$ is a closed hyperbolic surface then the center of $\mathcal{L}(F)$ is trivial. 
\end{style} 

It is natural to conjecture \cite[Open Problem 1]{Chas} that for a surface $F$ with non-empty boundary the center of $\mathcal{L}(F)$ is generated by the free homotopy classes of oriented closed curves which are either homotopic to boundary or homotopic to puncture. In this paper we prove this conjecture.

\begin{style}{Theorem 2} If $F$ is a hyperbolic surface of finite type with nonepmty geodesic  boundary then the center of $\mathcal{L}(F)$ is generated by the set of all free homotopy classes of oriented closed curves which are either homotopic to a boundary component or homotopic to a puncture.  
\end{style}

Goldman discovered this bracket operation while studying the Weil-Petersson symplectic form on Teichm{\"u}ller space. Using Wolpert's \cite{W} result on length and twist flow he showed that if Goldman bracket between two closed curve is zero and one of them has a simple representative then their geometric intersection number is zero. Later the combinatorial structure of $\mathcal{L}(F)$ has been studied. Using combinatorial topology, Chas \cite{Ch1}  proved a stronger version of Goldman's result namely if one of the curves has a simple representative then the number of terms in the Goldman bracket is same as their geometric intersection number. Chas and Krongold \cite{CK} proved that for a compact surface with non-empty boundary, $[\a,\a^3]$ determines the self-intersection number of $\a$. Using hyperbolic geometry, Chas and Gadgil \cite{GC} proved that there exists a positive integer $m_0$ such that for all $m\geq m_0$, the geometric intersection number between $\a$ and $\b$ is the number of terms in $[\a^m,\b]$ divided by $m$. There is a Lie cobracket defined by Turaev \cite{T} on $\Z(\mathcal{C})$ which is the dual object of Goldman bracket. This structure has been studied in \cite{CK0}, \cite{Ch}.
\vspace{3mm}

\noindent\textbf{Idea of the proof:}\hspace{1mm} Given an oriented surface of negative Euler characteristic, we fix a hyperbolic metric on it with geodesic boundary. Given two closed oriented curves $\a$ and $\b$ intersecting transversally we construct lifts of $(\a*_p\b)$ in $\H$ for each intersection point $p$. We show that the lifts are quasi-geodesics. Hence they are homotopic to unique geodesics. Therefore if two terms $(\a*_p\b)$ and $(\a*_q\b)$ cancel each other then corresponding geodesics will be same. 

Taking sufficiently higher power $m$ of $\a$ we ensure that if the geodesics are same then the quasi-geodesics are also same and hence the terms have same sign. Therefore there is no cancellation between the terms of $[\a^m,\b]$.

 Now if an element $\b=\sum_{i=1}^k\b_i$ of $\mathcal{L}(F)$ belongs to the center then we consider a simple closed curve $\a$ which intersects atleast one of the curves $\b_i$ non-trivially. Taking sufficiently higher power of $\a$ we ensure that the same terms of $[\a^m,\b_i]$ has the same sign. Also we show that if one term of $[\a^m,\b_i]$ and another term of $[\a^m,\b_j]$ are same then $\b_i$ and $\b_j$ are conjugates. 
 
 Therefore if $[\g,\b]$ is zero for all closed curve $\g$ then each $\b_i$ is disjoint from every simple closed curve and hence each $\b_i$ is either homotopic to a point or a boundary component or a puncture. A crucial fact we use in our proof is that the lifts of a simple closed geodesic in $\H$ are disjoint. 
\vspace{3mm}

\noindent\textbf{Organization of the paper:} In section 2 we recall some basic facts about hyperbolic geometry and hyperbolic surfaces. We also define the intersection number and loop product between two curves intersecting transversally (not necessarily in double points) and give an algebraic description of the geometric intersection number. In section 3 we mention two well known results about the fundamental group of a hyperbolic surface. In section 4  we recall the definition and some properties of Goldman bracket. In section 5 we construct the lifts of the terms in Goldman bracket and show that they are quasi-geodesics. In section 6 we prove the non-cancellation lemma which is the main lemma of this paper. In section 7 we prove the main theorems.

\noindent\textbf{Acknowledgements:} The author would like to thank Siddhartha Gadgil for his encouragement and enlightening conversations. The author would also like to thank Bidyut Sanki and Divakaran D. for all the discussions which helped him to understand the problem better.         

\section{Basics of hyperbolic geometry and hyperbolic surfaces}
In this section we recall some basic facts about hyperbolic geometry and hyperbolic surfaces. References for the results mentioned in this section are \cite{B}, \cite{Ka}, \cite{R}.

Let $F$ be a hyperbolic surface of finite type, i.e. $F$ is a surface of genus $g$ with $b$ boundary components and $n$ punctures such that, $2-2g-b-n< 0.$ Let $G=\pi_1(F)$ be the fundamental group of $F$. We identify $G$ with a discrete subgroup of $PSL_2(\R)$, the group of orientation preserving isometries of $\H$. The action of $G$ on $\H$ is properly discontinuous and without any fixed point. Therefore the quotient space is isometric to $F.$ Henceforth by an isometry of $\H$ we mean an orientation preserving isometry and by a closed curve we mean an oriented close curve.

A homotopically non-trivial closed curve in $F$ is called \emph{essential} if it is not homotopic to a puncture. By a \emph{lift} of a closed curve $\g$ to $\H$ we mean the image of a lift $\R\rightarrow \H$ of the map $\g\circ\pi$ where $\pi:\R\rightarrow S^1$ is the usual covering map. 

Every isometry $f$ of $\H$ belongs to one of the following classes:\\
1) \emph{Elliptic}: $f$ has a unique fixed point in $\H$.\\
2) \emph{Parabolic}: $f$ has a unique fixed point in $\partial\H.$\\
3) \emph{Hyperbolic}: $f$ has exactly two fixed points in $\partial\H.$ The geodesic joining these two fixed points is called the \emph{axis} of $f$ and is denoted by $A_f$. The isometry $f$ acts on $A_f$ by translation by a fixed positive number, called the \emph{translation number} of $f$ which we denote by $\t_f.$  

Since $G$ acts on $\H$ without any fixed point, $G$ does not contain any elliptic element. Essential closed curves in $G$ correspond to hyperbolic isometries and closed curves homotopic to punctures correspond to parabolic isometries.

Let $\a,\b\in G$ be two hyperbolic elements whose axes intersect at a point $P$ (Figure \ref{axis}). By \cite[Theorem 7.38.6]{B}, $\a\b$ is also hyperbolic. Let $Q$ be the point in $A_{\a}$ at a distance $\tau_{\a}/2$ from $P$ in the positive direction of $A_{\a}$ and $R$ be the point in $A_{\b}$ at a distance $\tau_{\b}/2$ from $P$ in the negative direction of $A_{\b}.$ Then the unique geodesic joining $R$ and $Q$ with orientation  from $R$ to $Q$ is the axis of $\a\b$ and the  distance between $Q$ and $R$ is $\tau_{\a\b}/2.$  

\begin{figure}[h]
  \centering
    \includegraphics[trim = 65mm 40mm 10mm 30mm, clip, width=15cm]{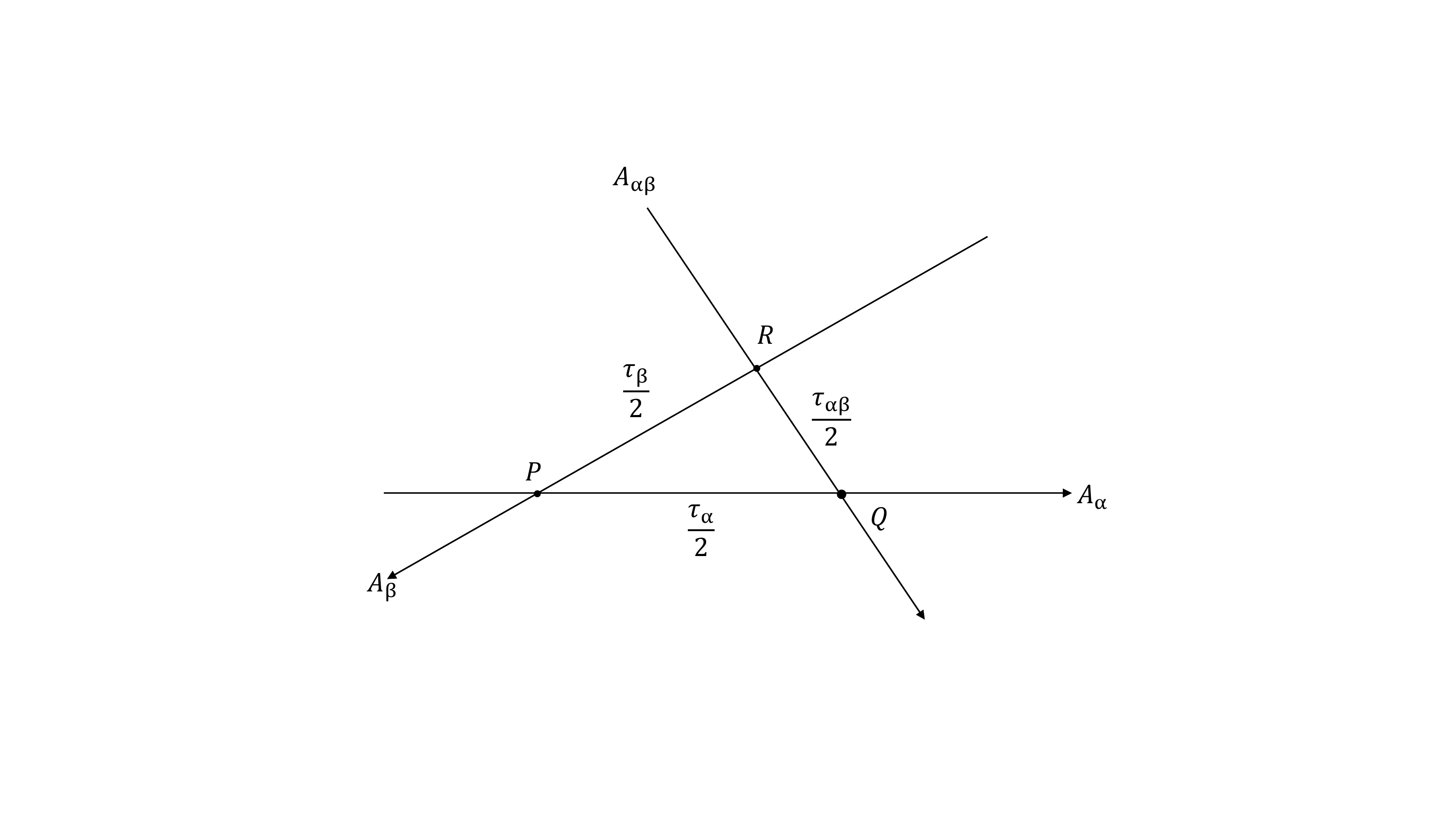}
     \caption{Axis of $\a\b$.}\label{axis}
\end{figure} 

There is a bijective correspondence between non-trivial conjugacy classes in $G$ and non-trivial free homotopy classes of closed curves in $F.$ Given a closed curve $\g$, we denote its free homotopy class by $\widetilde{\g}$ and we identify it by the corresponding conjugacy class in $G$ which we denote by $\la\g\ra$. Abusing notation we sometime denote the conjugacy class of $\g$ by $\g$ itself. Given $a,g\in G$, we denote $gag^{-1}$ by $a^g$. If $a$ is hyperbolic then $a^g$ is also hyperbolic with $\t_{a^{g}}=\t_a$ and $A_{a^{g}}=gA_a$ for all $g\in G$.  

Let $S^1=\{e^{2\pi i\theta}:\theta\in[0,1]\}$. We consider $S^1$ with counter-clockwise orientation. Let $\a:S^1\rightarrow F$ and $\b:S^1\rightarrow F$ be two closed curves. Given $\theta_1,\theta_2\in[0,1]$, we say $p(\theta_1,\theta_2)\in F$ is an \emph{intersection point} between $\a$ and $\b$ if $p(\theta_1,\theta_2)=\a(e^{2\pi i\theta_1})=\b(e^{2\pi\i\theta_2})$. We sometime omit $\theta_1$ and $\theta_2$ from the notation of intersection points. Denote the set of all intersection points between $\a$  and $\b$ by $\a\cap\b$.

Let $p=p(\theta_1,\theta_2)$ be an intersection point between $\a$ and $\b$. The \emph{loop product between $\a$  and $\b$ at $p$}, denoted by $(\a*_p\b)$ is defined to be 
$$(\a*_p\b)(e^{2\pi i\theta})=
\begin{cases}
\a(e^{2\pi i(\theta_1+2\theta)}) & \text{if } \theta\in[0,\frac{1}{2}]\\
\b(e^{2\pi(\theta_2+(2\theta-1))}) & \text{if } \theta\in[\frac{1}{2},1]. 
\end{cases}$$

The \emph{geometric intersection number} between free homotopy classes of closed curves $\a$ and $\b$, denoted by $i(\a, \b)$, is defined to be the minimal number of intersection points between a representative curve in the class $\widetilde{\a}$ and a representative curve in the class $\widetilde{\b}$ which intersects transversally.

Every free homotopy class of essential closed curve contains a unique closed geodesic whose length is same as the translation length of any element of the corresponding conjugacy class. By a slight abuse of notation we denote the free homotopy classes of essential closed curves by their geodesic representatives. Let $\la\a\ra$ and $\la\b\ra$ be two free homotopy classes of essential closed curves with geodesic representatives $\a$ and $\b$ respectively. Suppose they intersects each other. Then we can determine the geometric intersection number between $\la\a\ra$ and $\la\b\ra$ in the following way.

\begin{figure}[h]
  \centering
    \includegraphics[trim = 50mm 25mm 10mm 60mm, clip, width=15cm]{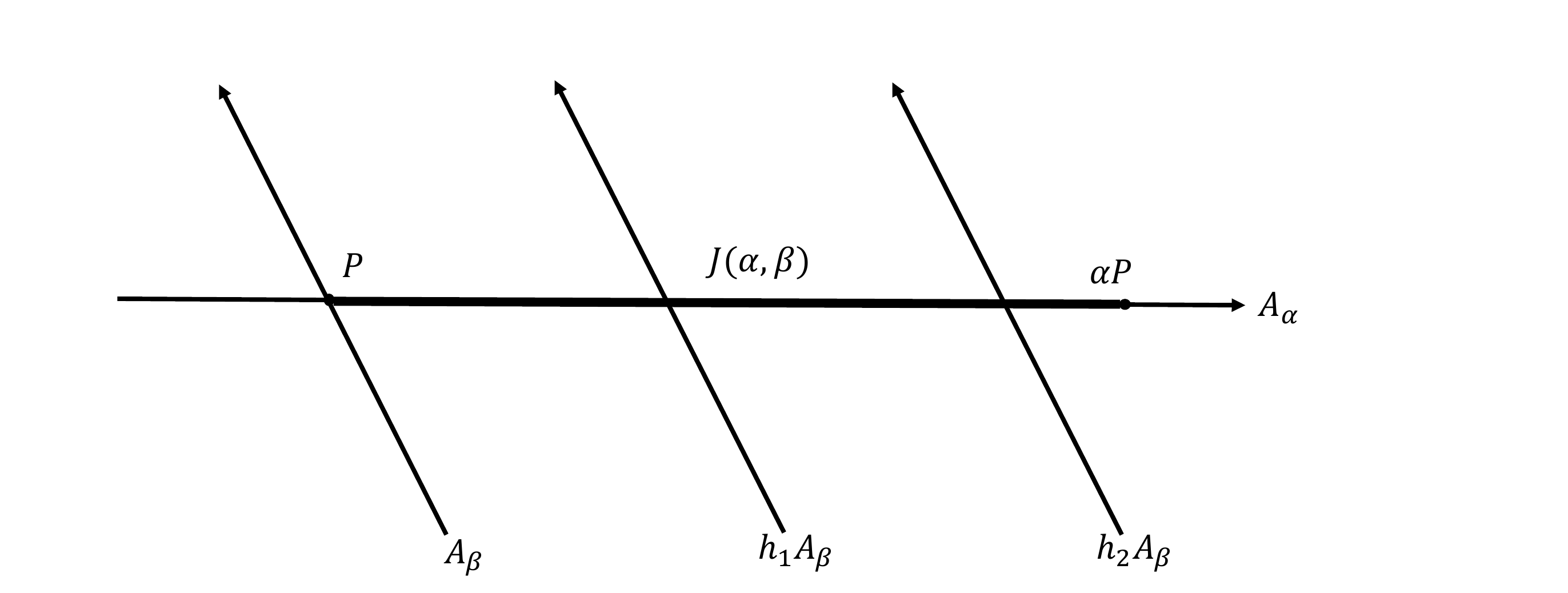}
     \caption{Lift of intersection points.}\label{intersection}
\end{figure}

Let $P$ be the intersection point between $A_{\a}$ and $A_{\b}$ where $A_{\a}$ and $A_{\b}$ are the axes of the hyperbolic transformations corresponding to  $\a$ and $\b$ respectively. Let $I(\a ,\b)$ be the segment of $A_{\a}$ of length $\t_{\a}$ starting from $P$ in the positive direction of $A_{\a},$ containing $P$ but not containing $\a P$ (Figure \ref{intersection}). For every intersection point $q\in\a\cap\b$, there exists a unique translate $gA_{\b}$ of $A_{\b}$ such that the intersection of $gA_{\b}$ with $I(\a,\b)$ is a lift $Q$ of $q$ and the angle between $A_{\a}$ and $gA_{\b}$ at $Q$ in their positive direction is same as the angle at $q$ between the arcs $(\a*_q\b)|_{[0,\frac{1}{4}]}$ and $(\a*_q\b)|_{[\frac{1}{2},\frac{3}{4}]}$. If $B$ denotes the cyclic subgroup generated by $\b$, define $J(\a, \b)=\{gB\in G/B : I(\a,\b)\cap gA_{\b}\neq \emptyset\}.$ Then cardinality of $J(\a ,\b)$ is exactly $i(\a ,\b)$.

\section{Definition and results about $\delta$-closeness}
\begin{definition} Suppose $A_1,A_2$ are two oriented geodesics in $\H$ containing points $p_1$ and $p_2$ respectively. Let $\Gamma$ be the oriented geodesic from $p_1$ to $p_2$ and $\theta_i$ be the angle between $A_i$ and $\Gamma$ at $p_i$ in the positive direction of both axes for $i\in\{1,2\}$. Given $\delta>0$, we say that $A_1$ and $A_2$ are \emph{$\delta$-close at $p_1$ and $p_2$} if $d(p_1,p_2)<\delta$ and $|\theta_1 -\theta_2|<\delta.$ 

We say that $A_1$ and $A_2$ are \emph{$\delta$-close} if there exist points $p_1$ and $p_2$ in $A_1$ and $A_2$ respectively, such that $A_1$ and $A_2$ are $\delta$-close at $p_1$ and $p_2$. 

The following lemma is a well known result.
\end{definition} 
\begin{lemma}\label{close}
For each $L>0$, there exists $\delta>0$ such that if $\a$ and $\b$ are hyperbolic elements in $G$ with $\tau_{\a}\leq L$, $\tau_{\b}\leq L$ and $A_{\a}$,$A_{\b}$ are $\delta$ close then the geodesics $A_{\a}$ and $A_{\b}$ coincide. 
\end{lemma}
\begin{proof} See \cite[Lemma 6.2]{GC}.   
\end{proof}
\begin{corollary}\label{nbd}
Given two nonzero positive number $L$ and $C$ there exists a constant $M>0$ such that for every pair of hyperbolic elements $\a$ and $\b$ in $G$ with $\t_{\a}\leq L$ and $\t_{\b}\leq L$, the set $$\{x\in A_{\a} : d(x,A_{\b})<C\}$$ is either empty or a geodesic segment of length at most $M$. 
\end{corollary}
\begin{proof} See \cite[Corollary 6.3]{GC}.
\end{proof}
\section{Goldman bracket: definition and properties}  
\begin{definition}
Let $\a$ and $\b$ be two closed curves in $F$ intersecting transversally. The Goldman bracket between $\a$ and $\b$, denoted by $[\a ,\b]$, is defined  to be $$[\a ,\b]=\sum_{p\in\a\cap\b}\epsilon_p\,\,(\widetilde{\a *_p\b})$$  where $\epsilon_p$ denotes the sign of the intersection between $\a$ and $\b$ at $p$ and $(\a *_p\b)$ denotes the loop product of $\a$ and $\b$ at $p.$ If $\a$ and $\b$ are disjoint then we define the bracket to be $0$.
\end{definition}
Note that the bracket operation is skew-symmetric, i.e. $[\a, \b]=-[\b, \a]$. Goldmann \cite{Gol} proved the following properties of the bracket operation.\\
1) It is well defined on the set of all free homotopy classes of closed curves, i.e. if $\a_1$ is freely homotopic to $\a_2$ and $\b_1$ is freely homotopic to $\b_2$ then $[\a_1 ,\b_1]=[\a_2 ,\b_2].$ \\ 
2) The bracket operation satisfies Jacobi identity, i.e. if $\a, \b ,\g$ are three closed curves which intersects with each other transversally then $$[\a ,[\b ,\g]]+[\b ,[\g ,\a]]+[\g ,[\a ,\b]]=0.$$
Let $\mathcal{C}$ be the set of all free homotopy classes of closed curves in $F$ and $\Z(\mathcal{C})$ be the free module generated by $\mathcal{C}$. Extend the bracket operation linearly to $\Z(\mathcal{C}).$ By the above properties the bracket is a Lie bracket on $\Z(\mathcal{C})$. We denote this Lie algebra by $\mathcal{L}(F)$.

If $\a$ is a closed curve freely homotopic to a puncture or to a point then for any closed curve $\b$, we find a closed curve $\widehat{\b}$ which is freely homotopic to $\b$ and disjoint from $\a$. Hence by the first property, $[\a,\b]=0.$ Therefore the Goldman bracket between a non-essential closed curve and any other closed curve is 0.
\section{Terms of Goldman bracket}
In this section we  describe the lifts of the terms of Goldman bracket and using the description we give an algebraic definition of Goldman bracket.

Let $\a$ and $\b$ be two essential geodesics and $p$ be an intersection point between $\a$ and $\beta$. Let $P$ be the lift of $p$ in $A_{\a}$ lying in the interval $I(\a,\b)$. There exists a unique $g\in G$ such that $P=A_{\a}\cap gA_{\b}=A_{\a}\cap A_{\b^g}$ and the angle between $A_{\a}$ and $gA_{\b}$ at $P$ in their positive direction is same as the angle at $p$ between the arcs $(\a*_p\b)|_{[0,\frac{1}{4}]}$ and $(\a*_p\b)|_{[\frac{1}{2},\frac{3}{4}]}$. A lift of $(\a*_p\b)$ passing through $P$ is a bi-infinite piecewise geodesic which has the following description. 

Let $h=\a\b^g$ and $\g:[0,1]\rightarrow \H$ be the concatenation of the geodesic arc from $P$ to $\a(P)$ and the geodesic arc from $\a(P)$ to $h(P)$ (Figure \ref{zigzag}). Define $\g(\a,\b^g):\R\rightarrow \H$ to be the periodic extension of $\g$ to $\R$, i.e. $\g(\a,\b^g)(t)=\g(t)$ for $t\in [0,1]$ and $\g(\a,\b)(t+k)=h^k\g(\a,\b^g)(t)$ for all $k\in\Z$. If we denote the geodesic arc from $P$ to $\a(P)$ by $I_{\a}^0$ and the geodesic segment from $\a(P)$ to $h(P)$ by $I_{\b^g}^0$ then by the description above $\g(\a,\b^g)$ consists of geodesic segments of the form $h^k(I_{\a}^0)$ and $h^k(I_{\b^g}^0)$ occurring alternatively.

 Denote $h^k(I_{\a}^0)$ by $I_{\a}^k$ and $h^k(I_{\b^g}^0)$ by $I_{\b^g}^k$. From the definition, the  length of $I_{\a}^k$ is $\t_{\a}$ and the length of $I_{\b^g}^k$ is $\t_{\b}$ for all $k\in\Z$. Hence by the description of the axis of the product of two isometries given in section 2,  $A_h$  intersects $I_{\a}^k$ and $I_{\b^g}^k$ in their midpoint for all $k\in \Z$.

\begin{definition}
Let $K\geq 1$ be a positive real number. A piecewise smooth curve $\g:\R\rightarrow\H$ is called \textit{K-quasi-geodesic} if for all $t_1,t_2\in \R$,$$l(\g|_{[t_1,t_2]})\leq K d(\g(t_1),(t_2))$$ where $l(\g)$ denotes the length of $\g$ and $d$ is the hyperbolic metric in $\H$.
\end{definition}

\begin{lemma}\label{quasi1}
Given $L> 0,$ there exist $K\geq 1$ and $C>0$  such that if $\a$ and $\b$ are two intersecting essential geodesics with $\t_{\a}\leq L$ and $\t_{\b}\leq L$ then for any $g\in G$ if $A_{\a}$ and $A_{\b^g}$ intersect then:\\
1) $\g(\a,\b^g)$ is a $K$-quasi-geodesic. \\
2) $\g(\a,\b^g)\subset N_{C/2}(A_h)$ and $A_h\subset N_{C/2}( \g(\a,\b^g))$, where $h=\a\b^g$ and $N_C(A_{h})$ denotes the $C$ neighbourhood of $A_{h}$.\\
3) $\g(\a,\b^g)$ is homotopic to $A_{h}.$
\end{lemma} 

 \begin{figure}[h]
  \centering
    \includegraphics[trim = 80mm 12mm 10mm 40mm, clip, width=13cm]{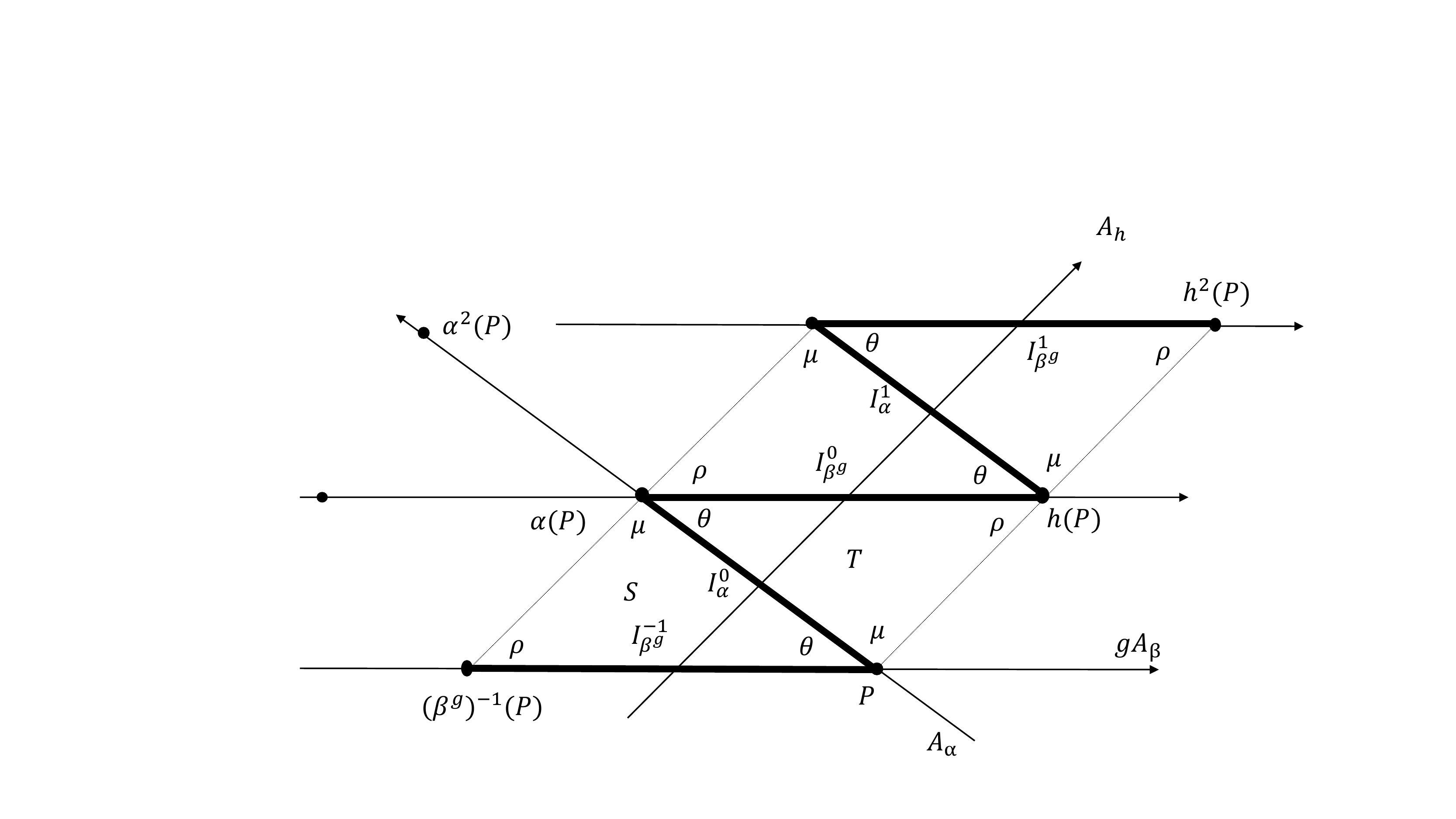}
     \caption{Lift of a term in Goldman bracket.}\label{zigzag}
\end{figure}

\begin{proof}
By Lemma \ref{close}, given $L>0$, there exists a $\delta>0$ such that the angle of intersection between $A_{\a}$ and $A_{\b^g}$ is bounded below by $\delta$ for all $g\in G$. Hence it is sufficient to prove that $\g(\a,\b^g)$ is a $K$-quasi-geodesic where $K$ depends only on the lower bound of the angle of intersection between $A_{\a}$ and $A_{\b^g}$. By the same argument it is enough to find a $C$ which depends on the lower bound of the angle of intersection.  

1) Let $T$ be the triangle with vertices $P,$ $\a(P)$ and $h(P)$ and $S$ be the triangle with vertices $P$, $\a(P)$ and $(\b^g)^{-1}(P)$ (Figure \ref{zigzag}). Let $\mu$, $\theta$ and $\rho$ be the angles of $T$ at $P$, $\a(P)$ and $h(P)$ respectively. Consider the polygonal region $\Delta=\cup_{k\in \Z}h^k(T\cup S)$. At each vertex of $\Delta$ the internal angles are $\mu$, $\theta$ and $\rho$ and their sum is at most $\pi.$ Hence $\Delta$ is convex. If $r$ and $s$ are any two points in $\g(\a,\b^g)$ then by the convexity of $\Delta$, the geodesic joining them lies inside $\Delta$. Hence it is enough to prove that $\g=\g(\a,\b^g)|_{[0,1]}$ is a quasi geodesic. By Toponogov comparison theorem \cite{Be}, $\g$ is a $K$ quasi geodesic where  
$$K=K(\theta) =
  \begin{cases}
   \frac{1}{\sin\theta}+\frac{1}{\tan\theta}+1 & \text{if } \theta\in(0,\pi /2)  \\
   \frac{1}{\sin\theta}+1       & \text{if } \theta\in[\pi/2,\pi).
  \end{cases} $$ Also from the definition of $K(\theta)$ we have that $K$ depends only on the lower bound of $\theta$.  
  
 2) See \cite[Lemma 7.2]{GC}.
 
 3) It follows from 1) and 2).
\end{proof}
The constants $C$ and $K$ mentioned above only depends on the lower bound of the angle of intersection between $A_{\a}$ and $A_{\b^g}$. Hence if we replace $\a$ by $\a^m$ then the same theorem holds for all $m\in\mathbb{N}.$  

From Lemma \ref{quasi1}, we have the following alternative algebraic definition of Goldman bracket. Let $\widetilde{\a}$ and $\widetilde{\b}$ be two free homotopy classes of closed curves. If they are disjoint then define the bracket to be zero. If one of them is non-essential then define the bracket to be zero. If both of them are essential and intersects then let $\a$ and $\b$ be the unique geodesics belong to the class $\widetilde{\a}$ and $\widetilde{\b}$ respectively. $\widetilde{\a}$ and $\widetilde{\b}$ corresponds to conjugacy classes of  hyperbolic elements of $G$. Define $$[\widetilde{\a},\widetilde{\b}]=\sum_{gB\in J(\a,\b)}\epsilon(\a,\b^g)\la\a\b^g\ra$$ where $\epsilon(\a,\b^g)$ denotes the sign of the intersection between $A_{\a}$ and $gA_{\b}$ at their point of intersection.  

\section{Non-Cancellation lemma}
Denote the length of a curve $\a$ by $l(\a)$. The following lemma is from \cite[Lemma 7.3]{GC}. We are including the proof for notational convenience.
\begin{lemma}\label{quasi2} Let $L, K$ and $C$ be  as in the Lemma \ref{quasi1}. For hyperbolic elements $\a , \b\in G$ with $\t_{a}\leq L$ and $\t_{\b}\leq L$, let $m$ be a positive integer such that $m\t_{\a}> 6KC.$  Then there exists a neighbourhood $U$ containing an open segment $J$ of $I_{\a^m}^0$ such that $l(J)\geq m\t_{\a}-6KC$, $U\subset N_C(I_{\a^m}^0)$,  $\overline{U}\cap N_C(I_{\a^m}^k)=\emptyset$ for all $k\neq 0$ and $\overline{U}\cap N_C(I_{\b}^k)=\emptyset$ for all $k\in Z$.
\end{lemma} 
\begin{proof}
Let $S$ and $R$ be points in $I_{\a^m}^0$ at distance $3KC$ from $P$ and $\a^mP$ (Figure \ref{u}). Let $A_s$ and $A_r$ be the geodesics perpendicular to $I^0_{\a^m}$ and passing through $S$ and $R$ respectively and $U$ be the subset of $N_C(I^0_{\a^m}) $ lying between $A_s$ and $A_r$. Let $J$ be the segment of $I^0_{\a^m}$ from $S$ to $R$. As $m\t_{\a}>6KC$, $J\neq\emptyset$ and $l(J)\geq m\t_{\a}-6KC$. 

Let $Z\in\overline{U}$ and let $Q\in N_C(I_{\b}^k)$  for some $k\in\Z.$ Let $Z_1$ be a point in $J$ whose distance from $Z$ is smaller than $C$ and $Q_1$ be a point in $I_{\b}^k$ whose distance from $Q$ is smaller than $C$. As $\g(\a^m,\b)$ is a $K$-quasi-geodesic, we have $$d(Q,Z)\geq d(Q_1,Z_1)-2C\geq d(P,S)/K-2C\geq C.$$ Hence $\overline{U}\cap N_C(I_{\b}^k)=\emptyset$ for all $k\in \Z$. The proof for $Q\in N_C(I_{\a^m}^k)$ is similar.  
\end{proof}

 \begin{figure}[h]
  \centering
    \includegraphics[trim = 80mm 45mm 80mm 90mm, clip, width=13cm]{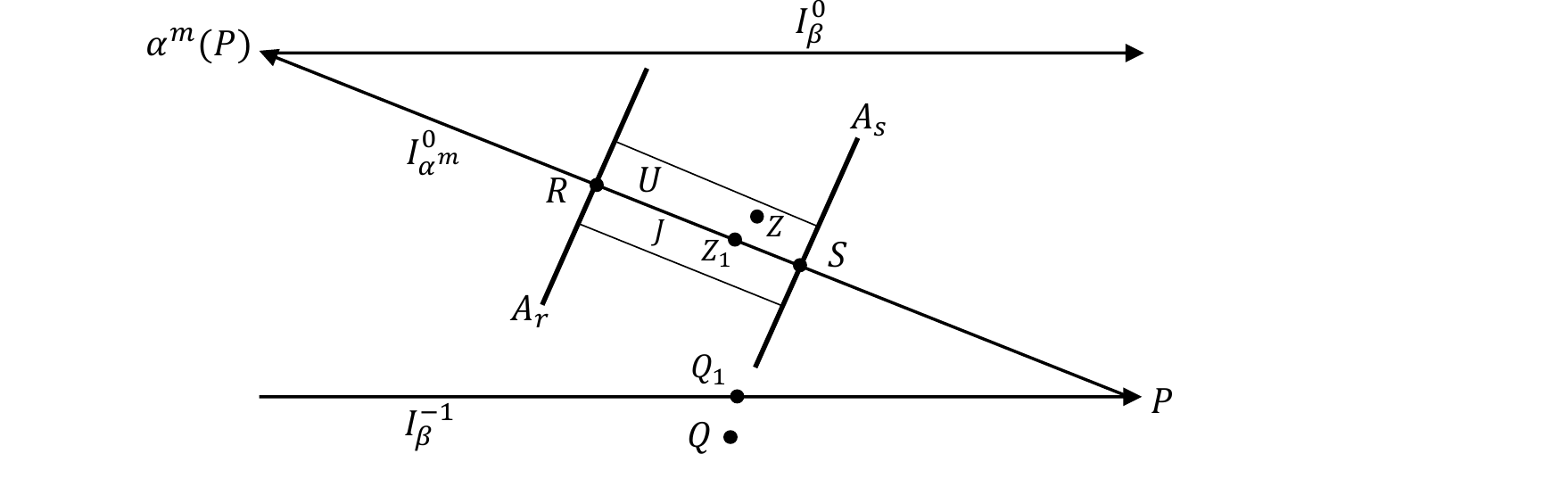}
     \caption{}\label{u}
\end{figure}

\begin{lemma}\label{nonca} Let $\a$ be a simple closed geodesic and $\a_1=(\a)^h$ for some $h\in G$. Let $\b$ and $\b_1$ be two distinct closed geodesics  whose axes are distinct and intersects the axis of $\a$. Let $L$ be a given positive number such that the  translation lengths of $\a,\b$ and $\b_1$ are bounded above by $L$. Then there exists $m_0$ such that for any $m> m_0$, $\g(\a^m,\b)=\g(\a_1^m,\b_1)$ whenever $\a^m\b=\a_1^m\b_1$. Moreover there exists  $u\in G$ such that $\a_1=\a^u$ and $\b_1=\b^u.$
\end{lemma}
\begin{proof}
 Let $\t_0$ be the systole, i.e. the length of a shortest length geodesic of $F.$
 
Let $C,K,M$ be the constants defined in Lemma \ref{quasi1} and Corollary \ref{nbd} respectively and $m>m_0$. Define $m_0=K(3M+10C)/\t_0.$

Since $\a^m\b=\a_1^m\b_1$, $A_{\a^m\b}=A_{\a_1^m\b_1}.$ Let $g=\a^m\b=\a_1^m\b_1.$ By Lemma \ref{quasi2},
\begin{equation}
\g(\a_1^m,\b_1)\subset N_{C/2}(A_g)\subset N_C(\g(\a^m,\b)).
\end{equation}

 Let $J$ and $U$ are as mentioned in Lemma \ref{quasi2} corresponding to $\g(\a^m,\b).$ Therefore $J\subset U,$ $J\subset I_{\a^m}^0\subset \g(\a^m,\b)$ and length$(J)\geq m\t_{\a}-6KC.$
 
\noindent\textbf{Claim 1:} $\g(\a_1^m,\b_1)$ intersects $U$ and does not intersect the boundary of $U$  contained in the boundary of $N_C(\g(\a^m,\b)).$

If $\g(\a_1^m,\b_1)$ does not intersect $U$, by $(1)$, $\second $ is contained in $N_C(\first)\setminus U=N_C(\first\setminus J)$ which is disconnected. Hence by $(1)$, $\second $ should intersect both components which contradicts that $\second $ is connected. Since $\second \subset N_C(\first )$, it does not intersect the boundary of $N_C(\first).$

Therefore the component of $U\cap\second$ consists of piecewise geodesic arcs starting and ending at the sides of $U$ of length $2C.$

\noindent\textbf{Claim 2:} Let $\mu$ be the component of $\second$ intersecting $U$. Then $\mu$ contains a geodesic segment $l$ of length greater than M.

\underline{Case-1:} Suppose $\mu$ contains more than three vertices. Then $\mu$ contains $I_{\a_{1}^m}^k$ for some $k\in \Z$ and length$(I_{\a_{1}^m}^k)=m\t_{\a_1}>m_0\t_{0}=K(3M+10KC)>M.$

\underline{Case-2:} Suppose $\mu$ contains at most three vertices. Then $\mu$ consists of at most three segments. Let $\nu$ be the longest segment with $l(\nu)=r.$ By hypothesis, $l(J)\geq m\t_{\a}-6KC>3KM+4KC$. Using triangle inequality and the properties of $m$ and $K$, we have $$(3M+4C)\leq K(3M+4C)< m\t_{\a}-6KC\leq 2C+3r+2C-6KC\leq 3r+4C.$$ Hence $r> M$ which proves claim 2. 

The geodesic segment $\nu$ is contained in $\second$ and $\nu\subset N_C(I_{\a^m}^0)$. Therefore by Corollary \ref{nbd}, $\nu$ intersects $I_{\a^m}^0$ in a geodesic segment. Hence $I_{\a^m}^0$ and $\second$ intersects in a geodesic segment.

\noindent\textbf{Claim 3:} If $\first$ and $\second$ intersects in a geodesic segment contained in $I_{\a^m}^0$ then they are equal, $\a_1=\a^u$ and $\b_1=\b^u$ for some $u\in G.$

Observe that for all $k\in\Z$, $l(I_{\a_1^m}^k)=m\t_{\a_1}=m\t_{\a}=l(I_{\a^m}^0)$ and $A_g$ intersects $I_{\a^m}^0$, $I_{\a_1^m}^k$ in their midpoints. Hence if $I_{\a^m}^0$ intersects $I_{\a_1^m}^k$ in a geodesic segment then they are equal. 

 If $I_{\a^m}^0$ intersects $I_{\b_1}^k$ in a geodesic segment for some $k\in\Z$, then by the construction of $\first$, $A_{\a}$ intersects $I_{\a_1^m}^{k+1}$ which lies in a translate of the geodesic $A_{\a_1}$ and hence in a translate of $A_{\a}$ (as $\a$ and $\a_1$ are conjugates). As $\a$ is simple, all translates of $A_{\a}$ are disjoint. Hence $I_{\a^m}^0$ cannot intersect $I_{\b_1}^k$ for any $k\in\Z$.
 
 \begin{figure}[h]
  \centering
    \includegraphics[trim = 10mm 35mm 0mm 50mm, clip, width=13.5cm]{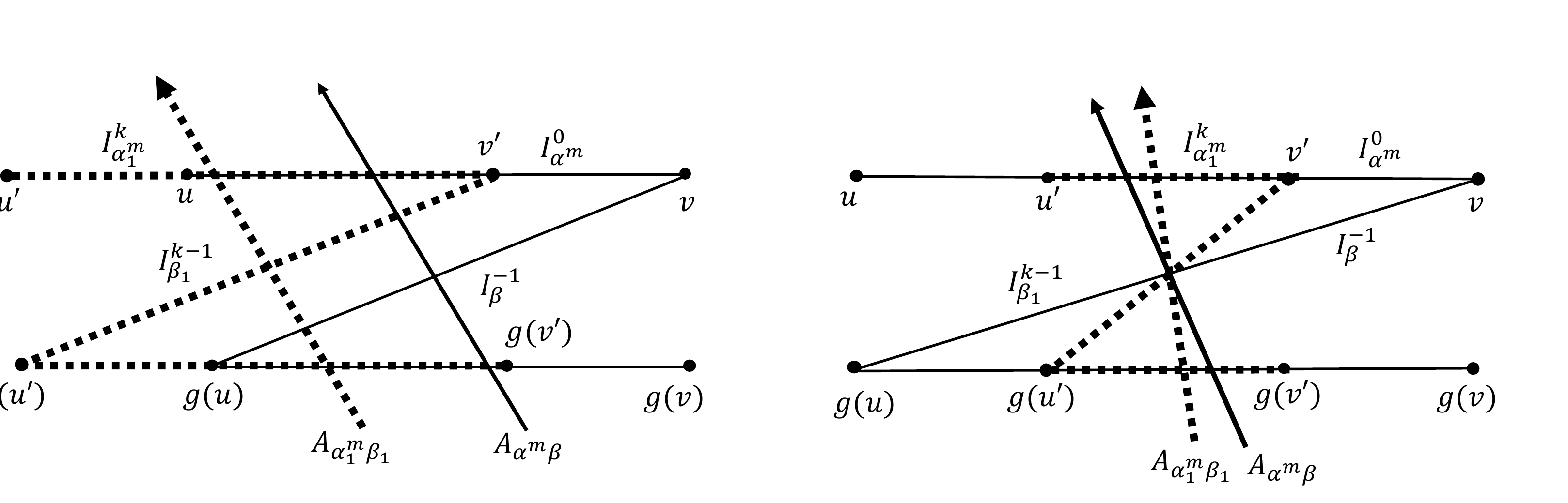}
     \caption{Possible intersections between $\first$ and $\second$}\label{share}
\end{figure}
 
 By Claim 3, $I_{\a^m}^0$ intersects $\second$ in a geodesic segment. Hence $I_{\a^m}^0$ intersects either  $I_{\a_1^m}^k$ or      
$I_{\b_1}^k$ in a geodesic segment. Since  $I_{\a^m}^0$ can not intersect $I_{\b_1}^k$ in a geodesic segment, $I_{\a^m}^0$ intersects $I_{\a_1^m}^k$ for some $k\in\Z$ and they are equal. Hence $I_{\a^m}^0=I_{\a_1^m}^k$. Since $I_{\b}^0$ and $I_{\b_1}^k$ are the unique geodesic segments joining the end point of  $I_{\a^m}^0=I_{\a_1^m}^k$ with the image of the starting point of $I_{\a^m}^0=I_{\a_1^m}^k$ under $g$, $I_{\b}^0=I_{\b_1}^k$. By the periodic property of the definition of $\first$ and $\second$ they are equal. Since $g^nI_{\a^m}^0=I_{\a_1^m}^0$ and  $g^nI_{\b}^0=I_{\b_1}^0$ for some $n$ taking $u=g^n$ we have $\a_1=\a^u$ and $\b_1=\b^u$. 
 \end{proof}

\section{Center of Goldman bracket}

\begin{lemma}\label{nonca1}
Let $F$ be a hyperbolic surface. Suppose $\a$ is an essential simple closed curve and $\b$ is an essential closed curve. Then there exists $m_0$ such that for all $m> m_0$, $[\a^m,\b]\neq 0.$  
\end{lemma}
\begin{proof}
Let $L=max\{\tau_\a,\tau_\b\}$ and $m_0$ be as in Lemma \ref{nonca}. If $m> m_0$ then $$[ \a^m,\b]=m\left(\sum_{kB\in J(\a^m,\b)}\epsilon(\a^m,\b^k){\la \a^m\b^k\ra}\right).$$ Suppose $\la \a^m\b^k\ra=\la \a^m\b^{k_1}\ra.$ Then for some $g\in G$ $$\a^m\b^k=(\a^m\b^{k_1})^g=(\a^m)^g(\b)^{k_1g}.$$ By Lemma \ref{nonca}, there exists $u\in G$ such that $\a$ is conjugate to $\a^g$  and $\b^k$ is conjugate to $\b^{k_1g}$ by the element $u$. Therefore $$\epsilon(\a^m,\b^k)=\epsilon((\a^{m})^u,(\b^{k})^u)=\epsilon(\a^{mg},\b^{k_1g})=\epsilon(\a^{m},\b^{k_1}).$$ Hence $[\a^m,\b]\neq 0.$   
   
\end{proof}  

The following theorem is a classical result.
\begin{theorem}\label{classical}
Let $F$ be a hyperbolic surface of finite type with geodesic boundary. Let $\g$ be a homotopically non-trivial closed curve whose geometric intersection number with any other non-trivial simple closed geodesic is zero. Then $\g$ is homotopic to boundary or puncture. 
\end{theorem}  

\begin{style}{Theorem 1}\label{final}
Let $F$ be a closed surface. Then the center of $\mathcal{L}(F)$ is trivial. 
\end{style} 

\begin{proof}
Let $\b=\sum_{i=1}^{n}c_i\b_i\in \mathcal{L}(F)$ where each $\b_i$ is a geodesic and $\b_i\neq \b_j$ for $i\neq j$. We show that if $\a$ is any simple closed curve which intersects at least one of the $\g_i$ then there exist $m>0$ such that $[\a^m,\sum_{i=1}^{n}c_i\b_i]$ is nonzero. Therefore if $\b$ belongs to the center of $\mathcal{L}(F)$ then each $\b_i$ is disjoint from every simple closed geodesic in $F$. Hence by Theorem \ref{classical}, each $\b_i$ is trivial and so is $\b$.

Let $\a$ be a simple closed geodesic which intersects atleast one $\b_i.$ If some $\b_k$ is disjoint from $\a$ then the Goldman bracket between $\a$ and $\b_k$ is zero, therefore without loss of generality  assume  that $\a$ intersects $\b_j$ for all $j\in\{1,2,\ldots ,n\}$. Let $L=max\{\tau_{\a_1},\tau_{\b_1},\tau_{\b_2},\ldots ,\tau_{\b_n}\}.$ Hence by Lemma \ref{nonca}, there exist $m_i$ for each $i\in\{1,2,\ldots ,n\}$ such that if $m> max\{m_i\}$ for all $i$ with $\a_1{^m}=(\a^m)^h$ for some $h\in G$ then $\gamma(\a^m,\b_i)=\gamma(\a_1^m,\b_j)$ whenever $\a^m\b_i=\a_1{^m}\b_j$. Also there exists  $g\in G$ such that $\a_1=\a^g$ and $\b_j=\b_i{^g}$. Since $[\a^m,\sum_{i=1}^{n}c_i\b_i]=\sum_{i=1}^{n}c_i[\a^m,\b_i]$, by Theorem \ref{nonca1}, it is enough to show that terms of  $[\a^m,\b_i]$ are distinct from the terms of $[\a^m,\b_j]$ for $i\neq j.$ 

Suppose  $\la \a^m\b_i^{k_i}\ra=\la \a^m\b_j^{k_j}\ra.$ Hence there exists $h\in G$ such that $$\a^m\b_i^{k_i}=(\a^m\b_j^{k_j})^h=(\a^m)^h(\b_j)^{k_jh}.$$ By Lemma \ref{nonca}, $\b_i^{k_i}$ and $\b_j^{k_j}$ are conjugates of each other in $G$. Therefore $\b_i$ and $\b_j$ are  freely homotopic to each other. Hence the geodesic representative corresponding to $\b_i$ and $\b_j$ are same, which contradicts the assumption.
\end{proof}

\begin{style}{Theorem 2}\label{final1}
Let $F$ be a hyperbolic surface of finite type with geodesic boundary. Then the center of $\mathcal{L}(F)$ is generated by the set of free homotopy classes of closed curves which are either homotopic to boundary or homotopic to puncture. 
\end{style}
\begin{proof}
Suppose $\g=\sum_{i=1}^nc_i\g_i$ belongs to the center of $\mathcal{L}(F)$. If the geometric intersection of any $\g_i$ with any simple closed curve $\a$ is non zero then by a similar argument as in Theorem 1 shows that $[\a,\g]\neq 0.$ Hence if $\g$ belongs to the center then the geometric intersection number of every $\g_i$ with any simple closed curve is zero. Thus by Theorem \ref{classical}, each $\g_i$  is homotopic to boundary or puncture.    
\end{proof}

\end{document}